\documentclass[11pt]{amsart}
\usepackage{amssymb}
\usepackage{amsmath}

\usepackage{color}

\title[Arc-transitive graphs of valency twice a prime]{Arc-transitive graphs of valency twice a prime admit a semiregular automorphism}
\date{\today}

\author[M. Giudici]{Michael Giudici}
\address{Michael Giudici, Department of Mathematics and Statistics\\ The University of Western Australia\\ 35 Stirling Highway\\ Crawley\\ WA 6009\\ Australia}
\email{michael.giudici@uwa.edu.au}
             
\author[G. Verret]{Gabriel Verret}
\address{Gabriel Verret, Department of Mathematics\\ The University of Auckland\\ Private Bag 92019\\ Auckland 1142\\ New Zealand}
\email{g.verret@auckland.ac.nz}

\newtheorem{theorem}{Theorem}[section]

\newtheorem{lemma}[theorem]{Lemma}
\newtheorem{corollary}[theorem]{Corollary}

\theoremstyle{remark}

\theoremstyle{definition}
\newtheorem{definition}[theorem]{Definition}

\newcommand{\C}{\mathrm C}

\newcommand{\V}{\mathrm V}

\newcommand{\K}{\mathrm K}
\newcommand{\M}{\mathrm M}
\newcommand{\N}{\mathrm N}
\newcommand{\Aut}{{\rm Aut}}
\newcommand{\Sym}{{\rm Sym}}
\newcommand{\Cos}{{\rm Cos}}
\newcommand{\PSL}{{\rm PSL}}
\newcommand{\PGL}{{\rm PGL}}
\newcommand{\SL}{{\rm SL}}
\newcommand{\GF}{{\mathrm {GF}}}

\renewcommand{\leq}{\leqslant}
\renewcommand{\geq}{\geqslant}

\begin{document}

\begin{abstract}
We prove that every finite arc-transitive graph of valency twice a prime admits a nontrivial semiregular automorphism, that is, a non-identity automorphism whose cycles all have the same length. This is a special case of the Polycirculant Conjecture, which states that all finite vertex-transitive digraphs admit such automorphisms.
\end{abstract}

\maketitle

\section{Introduction}

All graphs in this paper are finite. In 1981, Maru\v{s}i\v{c} asked if every vertex-transitive digraph admits a nontrivial semiregular automorphism~\cite{marusic}, that is, an automorphism whose cycles all have the same length. This question has attracted considerable interest and a generalisation of the affirmative answer is now referred to as the Polycirculant Conjecture \cite{seven}. One line of investigation of this question has been according to the valency of the graph or digraph. Every vertex-transitive graph of valency at most four admits such an automorphism \cite{DMMN,maruscap}, and so does every vertex-transitive digraph of out-valency at most three \cite{GMPV}. Every arc-transitive graph of prime valency  has a nontrivial semiregular automorphism \cite{MichaelXu} and so does every arc-transitive graph of valency 8 \cite{Verret}. Partial results were also obtained for arc-transitive graphs of valency a product of two primes \cite{JingXu}. We continue this theme by proving the following theorem.

\begin{theorem}\label{theo:main}
Arc-transitive graphs of valency twice a prime admit a nontrivial semiregular automorphism.
\end{theorem}

\section{Preliminaries}

If $G$ is a group of automorphisms of a graph $\Gamma$ and $v$ is a vertex of $\Gamma$, we denote by $G_v$ the stabiliser in $G$ of $v$, by $\Gamma(v)$ the neighbourhood of $v$, and by $G_v^{\Gamma(v)}$ the permutation group induced by $G_v$ on $\Gamma(v)$. We will need the following well-known results.

\begin{lemma}\label{lem:locallyreg}
Let $\Gamma$ be a connected graph and $G\leqslant\Aut(\Gamma)$ such that all orbits of $G$ on $\V(\Gamma)$ have the same size. If a prime $p$ divides  $|G_v|$ for some $v\in\V(\Gamma)$, then $p$ divides $|G_v^{\Gamma(v)}|$.
\end{lemma}

\begin{lemma}\label{lemma:coprime}
Let $G$ be a permutation group and let $K$ be a normal subgroup of $G$ such that $G/K$ acts faithfully on the set of  $K$-orbits. If $G/K$ has a semiregular element $Kg$ of order $r$ coprime to $|K|$, then $G$ has a semiregular element of order $r$.
\end{lemma}
\begin{proof}
See for example \cite[Lemma~2.3]{Verret}.
\end{proof}

\begin{lemma}\label{lemma:primepowerdegree}
A transitive group of degree a power of a prime $p$ contains a semiregular element of order $p$.
\end{lemma}
\begin{proof}
In a transitive group of degree a power of a prime $p$, every Sylow $p$-subgroup is transitive. A non-trivial element of the center of this subgroup must be semiregular.
\end{proof}

Recall that a permutation group is \emph{quasiprimitive} if every non-trivial normal subgroup is transitive, and \emph{biquasiprimitive} if it is not quasiprimitive and every non-trivial normal subgroup has at most two orbits.

\section{Arc-transitive graphs of prime valency}

In the most difficult part of our proof, the arc-transitive graph of valency twice a prime will have a normal quotient with prime valency. We will thus need a lot of information about arc-transitive graphs of prime valency, which we collect in this section. We start with the following result, which is  \cite[Theorem~5]{BurnessMichael}:

\begin{theorem}\label{theo:submain}
Let $\Gamma$ be a connected $G$-arc-transitive graph of prime valency $p$ such that the action of $G$ on $\V(\Gamma)$ is either quasiprimitive
or biquasiprimitive. Then one of the following holds:
\begin{enumerate}
\item $G$ contains a semiregular element of odd prime order;
\item $|\V(\Gamma)|$ is a power of $2$;
\item $\Gamma = \K_{12}$, $G = \M_{11}$ and $p = 11$;
\item $|\V(\Gamma)| = (p^2 -1)/2s$ and $G = \PSL_2(p)$ or $\PGL_2(p)$, where $p$ is a Mersenne prime and $\C_r \leq \C_s < \C_{(p-1)/2}$ , where $r$ is the product of the distinct prime divisors of $(p-1)/2$;
\item $|\V(\Gamma)| = (p^2 - 1)/s$ and $G = \PGL_2(p)$, where $p$ and $s$ are as in part (4), and $G$ is the standard double cover of the graph given in (4).
\end{enumerate}
\end{theorem}

We note that in Cases (4) and (5), we must have $p\geq 127$, since this is the smallest Mersenne prime $p$ such that $(p-1)/2$ is not squarefree. This fact will be used at the end of Section~\ref{sec:main}.

\begin{corollary}\label{cor:pvalent}
Let $\Gamma$ be a connected $G$-arc-transitive graph of prime valency. Then one of the following holds:
\begin{enumerate}
\item $G$ contains a semiregular element of odd prime order;
\item $|\V(\Gamma)|$ is a power of $2$;
\item $G$ contains a normal $2$-subgroup $P$ such that $(\Gamma/P,G/P)$ is one of the graph-group pairs in (3--5) of Theorem~\ref{theo:submain}.
\end{enumerate}
\end{corollary}
\begin{proof}
Suppose that $|\V(\Gamma)|$ is not a power of $2$ and let $P$ be a normal subgroup of $G$ that is maximal subject to having at least three orbits on $\V(\Gamma)$. In particular, $P$ is the kernel of the action of $G$ on the set of $P$-orbits. Hence $G/P$  acts faithfully, and quasiprimitively or biquasiprimitively on $\V(\Gamma/P)$. Since $\Gamma$ has prime valency and $G$-arc-transitive, \cite[Theorem 9]{Lorimer} implies that $P$ is semiregular.  We may thus assume that $P$ is a $2$-group. (Otherwise $P$ contains a semiregular element of odd prime order.) If  $G/P$ contains a semiregular element of odd prime order, then Lemma~\ref{lemma:coprime} implies that so does $G$. We may assume that this is not the case. Similarly, we may assume that $|\V(\Gamma/P)|$ is not a power of $2$. (Otherwise, $|\V(\Gamma)|$ is a power of $2$.) It follows that $\Gamma/P$ and $G/P$ are as is  (3--5) of Theorem~\ref{theo:submain}.
\end{proof}

We will now prove some more results about the graphs that appear in (3--5) of Theorem~\ref{theo:submain}. Let us first recall the notion of \emph{coset graphs}. Let $G$ be a group with a subgroup $H$ and let $g\in G$ such that $g^2\in H$ but $g\notin \N_G(H)$. The graph $\Cos(G,H,HgH)$ has vertices  the right cosets of $H$ in $G$, with two cosets $Hx$ and $Hy$  adjacent if and only if $xy^{-1}\in HgH$. Observe that the action of $G$ on the set of vertices by right multiplication induces an arc-transitive group of automorphisms such that $H$ is the stabiliser of a vertex. Moreover, every arc-transitive graph can be constructed in this way \cite{Sabidussi}.

\begin{lemma}\label{lem:triangle}
The graphs in (3) and (4) of Theorem~\ref{theo:submain} have a triangle.
\end{lemma}
\begin{proof}
Clearly $K_{12}$ has a triangle so suppose that $\Gamma$ is one of the graphs given in (4). Since $G$ is quasiprimitive on $\V(\Gamma)$ and $\Gamma$ has prime valency, it follows that $\PSL_2(p)$ is arc-transitive on $\Gamma$ and so $\Gamma=\Cos(G,H,HgH)$ where $G=\PSL_2(p)$, $H\cong \C_p\rtimes \C_s$ and $g\in G\backslash \N_G(H)$ such that $HgH$ is a union of $p$ distinct right cosets of $H$. 

Let $x\in \N_G(H)\cong \C_p\rtimes \C_{(p-1)/2}$. Note that $G$ acts 2-transitively on the set of right cosets of $\N_G(H)$ with the stabiliser of any two points being isomorphic to $\C_{(p-1)/2}$.  Now the coset $Hx\in \V(\Gamma)$ is fixed by $H$ and in particular $H$ has $|\N_G(H):H|=(p(p-1)/2)/ps=(p-1)/2s$ fixed points on $\V(\Gamma)$. If $y\notin \N_G(H)$, then the orbit of $Hy$ under $H$ has length $|H:H\cap H^y|$ and since  $|\N_G(H): (\N_G(H)\cap \N_G(H)^y)|=p$, it follows that  $|H:H\cap H^y|=p$. Thus the points of $\V(\Gamma)$ that are not fixed by $H$  are permuted by $H$ in orbits of size $p$ and so for any $g\notin \N_G(H)$ we have that $HgH$ is a union of $p$ distinct right cosets of $H$.

For each $x\in \N_G(H)$ define the bijection $\lambda_x$ of $\V(\Gamma)$ by $Hy\mapsto x^{-1}Hy=Hx^{-1}y$.  Since $G$ acts on $\V(\Gamma)$ by right multiplication, we see that $\lambda_x$ commutes with each element of $G$. Moreover, $\lambda_x$  is nontrivial if and only if $x\notin H$.
Let  $C=\{\lambda_x\mid x\in \N_G(H)\}\leqslant \Sym(\V(\Gamma))$.  Since $G$ acts transitively on $\V(\Gamma)$ and $C$ centralises $G$, it follows from \cite[Theorem 4.2A]{DM}  that $C$ acts semiregularly on $\V(\Gamma)$. Now $|C|=|\N_G(H):H|=(p-1)/2s$ and $G\times C\leqslant \Sym(\V(\Gamma))$. Since $C\trianglelefteq G\times C$, the set of  orbits of $C$ forms a system of imprimitivity  for $G\times C$ and hence for $G$. Moreover, since $C$ is semiregular,  comparing orders yields that $C$ has  $|\V(\Gamma)|/|C|=p+1$ orbits. One of these orbits is the set of fixed points of $H$ and $H$ transitively permutes the remaining $p$ orbits of $C$. In particular, it follows that $C$ transitively permutes the nontrivial orbits of $H$ and so the isomorphism type of $\Gamma$ does not depend on the choice of the double coset $HgH$. 

Let $Z$ be the subgroup of scalar matrices in $\SL_2(p)$ and let $\hat{H}$ be the subgroup of the  stabiliser in $\SL_2(p)$ of the 1-dimensional subspace $\langle (1,0)\rangle$ such that $H=\hat{H}/Z$. Note that $\hat{h}= \begin{pmatrix}1&0\\1&1\end{pmatrix}\in \hat{H}$ and let $\hat{g}=\begin{pmatrix}0&1\\-1&0\end{pmatrix}$.
In particular, letting $h=\hat{h}Z$ and $g=\hat{g}Z$ we have that $g\notin \N_G(H)$ and so we may assume that $\Gamma=\Cos(G,H,HgH)$. Now $Hg$ and $Hgh$ are both adjacent to $H$ and one can check that $g(gh)^{-1}=hgh\in HgH$ and so $\{H,Hg,Hgh\}$ is a triangle in $\Gamma$.
\end{proof}

\begin{definition}\label{def:dense}
Let $\Gamma$ be a graph and let $S_0$ be a subset of $\V(\Gamma)$. Let $S=S_0$. If a vertex $u$ outside $S$ has at least two neighbours in $S$, add $u$ to $S$. Repeat this procedure until no more vertex outside $S$ have this property. If at the end of the procedure, we have $S=\V(\Gamma)$, then we say that $\Gamma$ is \textit{dense with respect to $S_0$}.
\end{definition}

\begin{corollary}\label{cor:dense}
Let $\Gamma$ be a graph in (3) or (4) of Theorem~\ref{theo:submain} and let $S_0=\{u,v\}$ be an edge of $\Gamma$. Then $\Gamma$ is dense with respect to $S_0$.
\end{corollary}
\begin{proof}
Since $\Gamma$ is arc-transitive of prime valency $p$, the local graph at $v$ (that is, the subgraph induced on $\Gamma(v)$) is a vertex-transitive graph of order $p$ and thus vertex-primitive. By Lemma~\ref{lem:triangle}, $\Gamma$ has a triangle so the local graph has at least one edge and thus must be connected. It follows that, running the process described in Definition~\ref{def:dense} starting at $S=S_0$, eventually $S$ will contain all neighbours of $v$. Repeating this argument and using connectedness of $\Gamma$ yields the desired conclusion.
\end{proof}

The following is immediate from Definition~\ref{def:dense}.

\begin{lemma}\label{lem:dense2}
Let $\Gamma$ be a graph and $S_0$ be a set of vertices such that $\Gamma$ is dense with respect to $S_0$. Then the standard double cover of $\Gamma$, with vertex-set $\V(\Gamma)\times\{0,1\}$, is dense with respect to $S_0\times \{0,1\}$.
\end{lemma}

\section{Proof of Theorem~\ref{theo:main}}\label{sec:main}
Let $p$ be a prime, let $\Gamma$ be an arc-transitive graph of valency $2p$ and let $G=\Aut(\Gamma)$. We may assume that $\Gamma$ is connected. If $G$ is quasiprimitive  or bi-quasiprimitive, then $G$ contains a nontrivial semiregular element, by  \cite[Theorem 1.1]{Michael} and \cite[Theorem~1.4]{MichaelXu}. We may thus assume that $G$ has a minimal normal subgroup $N$ such that $N$ has at least three  orbits. In particular, $\Gamma/N$ has valency at least $2$.

If $N$ is nonabelian, then $G$ has a nontrivial semiregular element by \cite[Theorem 1.1]{JingXu}. We may therefore assume that $N$ is abelian and, in particular, $N$ is an elementary abelian $q$-group for some prime $q$. 

We may also assume that $N$ is not semiregular that is, $N_v\neq 1$ for some vertex $v$. It follows  from Lemma \ref{lem:locallyreg} that $1\neq N_v^{\Gamma(v)}\trianglelefteq G_v^{\Gamma(v)}$. As $\Gamma$ is $G$-arc-transitive, we have that $G_v^{\Gamma(v)}$ is transitive and so the orbits of $N_v^{\Gamma(v)}$ all have the same size, either $2$ or $p$. Since $N$ is a $q$-group, this size is equal to $q$.  Writing $d$ for the valency of  $\Gamma/N$, we have that either $(d,q)=(2,p)$ or $(d,q)=(p,2)$.

If $d=2$ and $q=p$, then it follows  from \cite[Theorem~$1$]{PraegerXu} that $\Gamma$ is isomorphic to a graph denoted by $\C(p,r,s)$ in  \cite{PraegerXu}. By \cite[Theorem~$2.13$]{PraegerXu},  $\Aut(\C(p,r,s))$ contains the nontrivial semiregular automorphism $\varrho$ defined in \cite[Lemma~$2.5$]{PraegerXu}.

We may thus assume that $d=p$ and $q=2$. In this case, if $u$ is adjacent to $v$, then $u$ has exactly $2=2p/d$ neighbours in $v^N$. Let $K$ be the kernel of the action of $G$ on $N$-orbits. By the previous observation, the orbits of $K_v^{\Gamma(v)}$ have size $2$ and thus it is a $2$-group. It follows from Lemma \ref{lem:locallyreg} that $K_v$ is a $2$-group and thus so is $K=NK_v$. 

Now, $G/K$ is an arc-transitive group of automorphisms of $\Gamma/N$, so we may apply Corollary~\ref{cor:pvalent}. If $G/K$ has a semiregular element of odd prime order, then so does $G$, by Lemma~\ref{lemma:coprime}. If $|\V(\Gamma/N)|$ is a power of $2$, then so is $|\V(\Gamma)|$ and, in this case, $G$ contains a semiregular involution by Lemma~\ref{lemma:primepowerdegree}. We may thus assume that we are in case (3) of Corollary~\ref{cor:pvalent}, that is, $G/K$ contains a normal $2$-subgroup $P/K$ such that $(\Gamma/P,G/P)$ is one of the graph-group pairs in (3--5) of Theorem~\ref{theo:submain}. Note that $P$ is a $2$-group. Let $M$ be a minimal normal subgroup of $G$ contained in the centre of $P$. We may assume that $M$ is not semiregular hence $M_v\neq 1$ and so by Lemma \ref{lem:locallyreg}, $M_v^{\Gamma(v)}\neq 1$. Moreover, $|M|\neq 2$ as otherwise $M_v=M$ and we would deduce that $M$ fixes each element of $\V(\Gamma)$, a contradicition. Since $M$ is central in $P$, $M_v$ fixes every vertex in $v^P$.

Note that the $G$-conjugates of $M_v$ must cover $M$, otherwise $M$ contains a nontrivial semiregular element. By the previous paragraph, the number of conjugates of $M_v$ is bounded above by the number of $P$-orbits, that is $|\V(\Gamma/P)|$, so we have 
$$|M|\leq |M_v||\V(\Gamma/P)|.$$

Since $\Gamma$ is connected and $G$-arc-transitive, there are no edges within $P$-orbits. As $M_v^{\Gamma(v)}\neq 1$, there exists $g\in M_v$ such that $w$ and $w^g$ are distinct neighbours of $v$. Let $u$ be the other neighbour of $w$ in $v^P$. Since $M_v$ fixes every element of $v^P$ it follows that $u$ is also a neighbour of $w$ and $w^g$ and so $\{v,w,u,w^g\}$ is a 4-cycle in $\Gamma$. Thus the graph induced between adjacent $P$-orbits is a union of $\C_4$'s. 

If $x$ is a vertex and $y^P$ is a $P$-orbit adjacent to $x$, then there is a unique $\C_4$ containing $x$ between $x^P$ and $y^P$, and thus a unique vertex $z$ antipodal to $x$ in this $\C_4$. We say that $z$ is the \textit{buddy of $x$ with respect to $y^P$}. The set of buddies of $v$ is equal to $\Gamma_2(v)\cap v^P$, which is clearly fixed setwise by $G_v$. Moreover, each vertex has the same number of buddies. Furthermore, since $G_v$ transitively permutes the set of $p$ $P$-orbits adjacent to $v^P$, either $v$ has a unique buddy or has exactly $p$ buddies.

 If $v$ has a unique buddy $z$, then $\Gamma(v)=\Gamma(z)$, and so swapping every vertex with its unique buddy is a nontrivial semiregular automorphism.  Thus it remains to consider the case where $v$ has $p$ buddies. We first prove the following.

\smallskip

\textbf{Claim:} If $X$ is a subgroup of $M$ that fixes pointwise both $a^P$ and $b^P$, and $c^P$ is a $P$-orbit adjacent to both $a^P$ and $b^P$, then $X$ fixes $c^P$ pointwise.

\textsc{Proof:} Suppose that some $x\in X$ does not fix $c$. Now $x$ fixes $a^P$ pointwise, so $c^x$ must be the buddy of $c$ with respect to $a^P$. Similarly, $c^x$ must be the buddy of $c$ with respect to $b^P$. These are distinct, which is a contradiction. It follows that $X$ fixes $c$ and, since $X\leq M$, also $c^P$.

\smallskip

Let $s\geq 1$, let $\alpha=(v_0,\ldots,v_s)$ be an $s$-arc of $\Gamma$ and let $\alpha'=(v_0,\ldots,v_{s-1})$. Now $|{v_s}^{M_{v_{s-1}}}|=2$, so $|M_{v_{s-1}}:M_{v_{s-1}v_s}|=2$ and $|M_{\alpha'}:M_\alpha|\leq 2$. Applying induction yields that 
\begin{align}\label{arceq}
|M_{v_0}:M_\alpha|\leq 2^s.
\end{align}

We first assume that $\Gamma/P$ and $G/P$ are as in (3) or (4) of Theorem~\ref{theo:submain}. Let $\{u,v\}$ be an edge of $\Gamma$.  By the previous paragraph, we have $|M_v:M_{uv}|\leq 2$. Recall that $M_v$ fixes all vertices in $v^P$, so $M_{uv}$ fixes all vertices in $v^P\cup u^P$.  Combining the claim with Corollary~\ref{cor:dense} yields that $M_{uv}=1$ and thus $|M_v|=2$. It follows that $|M|\leq |M_v||\V(\Gamma/P)|$ so $|M|\leq 2|\V(\Gamma/P)|$. On the other hand, $M$ is an irreducible $(G/P)$-module over $\GF(2)$ of dimension at least two.  Since $G/P$ is nonabelian simple or has a nonabelian simple group as an index two subgroup, this implies that $M$ is a faithful irreducible $(G/P)$-module over $\GF(2)$. If $G/P=\M_{11}$, then $|M|\geq 2^{10}$ \cite{James}, contradicting $|M|\leqslant 2\cdot 12=24$. If $G/P=\PSL_2(p)$ or $\PGL_2(p)$ then by \cite{Burkhardt}, $|M|\geq 2^{(p-1)/2}$.  Recall that $p\geq 127$ and  so this contradicts $|M|\leqslant 2(p^2-1)/2s<p^2-1$.

Finally, we assume that $\Gamma/P$ is in (5) of Theorem~\ref{theo:submain}, that is, $\Gamma/P$ is the standard double cover of a graph $\Gamma'$ which appears in (4) of Theorem~\ref{theo:submain}. In particular, $\V(\Gamma/P)=\V(\Gamma')\times \{0,1\}$. By Lemma~\ref{lem:triangle}, $\Gamma'$ has a triangle, say $\{u,v,w\}$. By Corollary~\ref{cor:dense} and Lemma~\ref{lem:dense2}, $\Gamma/P$ is dense with respect to $\{u,v\}\times \{0,1\}$. Now, let  $\overline{\alpha}=((u,0),(v,1),(w,0),(u,1),(v,0))$. Since $\overline{\alpha}$ contains $\{u,v\}\times \{0,1\}$, $\Gamma/P$ is dense with respect to $\overline{\alpha}$. Note that $\overline{\alpha}$ is a $4$-arc of $\Gamma/P$. Let $\alpha$ be a $4$-arc of $\Gamma$ that projects to $\overline{\alpha}$. Since $\Gamma/P$ is dense with respect to $\overline{\alpha}$, arguing as in the last paragraph yields $M_\alpha=1$. On the other hand, if $v\in\V(\Gamma)$ is the the initial vertex of $\alpha$, then  by \eqref{arceq}, we have $|M_v:M_\alpha|\leq 2^4$ and thus $|M_v|\leq 2^4$. Since $|M|\leqslant |M_v||\V(\Gamma/P)|$ it follows that $|M|\leqslant 2^4(p^2-1)/s$.  As above, $M$ is a faithful irreducible $(G/P)$-module over $\GF(2)$ of dimension at least two. Since $G/P=\PGL_2(p)$ we have from \cite{Burkhardt} that $|M|\geqslant 2^{(p-1)/2}$, which again contradicts $|M| \leqslant 2^4(p^2-1)/s<2^4(p^2-1)$.

%%%%%%%%%%%%%%%%%%%%%%%%%%%%%%%%%%%%%%%%%%%%%%%%%%%%%%%%%%%%%%%%%%%%%%%%%%%%%%%
\bibliographystyle{amsplain}

\end{document}